\documentclass[11pt, reqno]{amsart}
\usepackage[utf8]{inputenc}
\usepackage{amsmath}
\usepackage{amsthm}
\usepackage{amsfonts}
\usepackage{amssymb}
\usepackage{mathrsfs, mathtools, mathdots}
\usepackage{diagbox}
\usepackage{relsize}
\DeclarePairedDelimiter \floor {\lfloor} {\rfloor}

\usepackage{enumerate, xcolor, graphicx}
\definecolor{Magenta}{cmyk}{0,1,0,0}
\definecolor{dgreen}{rgb}{0.,0.6,0.}

\usepackage[top=25mm,bottom=25mm,left=27mm,right=27mm]{geometry}

\newtheorem{theorem}{Theorem}[section]
\newtheorem{proposition}{Proposition}[section]
\newtheorem{conjecture}{Conjecture}[section]

\newtheorem*{definition*}{Definition}
\newtheorem*{theorem*}{Theorem}
\newtheorem*{remark*}{Remark}
\newtheorem*{problem*}{Problem}
\newtheorem*{conjecture*}{Conjecture}
\newtheorem{lemma}{Lemma}[section]

\newcommand*{\De}{\Delta}

\newcommand*{\mkl}{M_{k,\ell}}
\newcommand*{\dkl}{\De_{k,\ell}}
\newcommand*{\akl}{A_{k,\ell}}
\newcommand*{\gkl}{G_{k,\ell}}
\newcommand*{\hkl}{H_{k,\ell}}
\newcommand*{\pkl}{P_{k,\ell}}
\newcommand*{\fkl}{f_{k,\ell}}
\newcommand*{\lp}{\left(}
\newcommand*{\rp}{\right)}

\newcommand*{\de}{\delta}
\renewcommand*{\th}{\theta}
\newcommand*{\vt}[1]{\left\lvert #1 \right\rvert}
\newcommand*{\abs}[1]{\lvert#1\rvert}

\begin{document}
	
	\title[Zeros of a variant of Ramanujan polynomials]{Variant of Ramanujan polynomials and \\ a conjecture of Maji \& Sarkar}
	
	\author[M. Charan]{Mrityunjoy Charan}
	
	\author[J. Meher]{Jaban Meher}
	
	\author[S. Pathak]{Siddhi Pathak}
	
	\address{(M. Charan) The Institute of Mathematical Sciences, A CI of Homi Bhabha National Institute, CIT Campus, Taramani, Chennai, Tamil Nadu, India 600113.}
	
	\email{mcharan@imsc.res.in}
	
	\address{(J. Meher) School of Mathematical Sciences, National Institute of Science Education and Research, Bhubaneswar, An OCC of Homi Bhabha National Institute, P. O. Jatni, Khurda, Odisha, India 752050.}

	\email{jaban@niser.ac.in}

	\address{(S. Pathak) Chennai Mathematical Institute, H-1 SIPCOT IT Park, Siruseri, Kelambakkam, Tamil Nadu, India 603103.}

	\email{siddhi@cmi.ac.in}
	
	\date{\today}


	
	\thanks{Research of the third author was partially supported by an INSPIRE Faculty fellowship.}

	\begin{abstract}
		We settle a conjecture proposed by B. Maji and T. Sarkar regarding the location of zeros of a two-parameter family of reciprocal polynomials, $R_{k,\ell}(z)$ for positive integers $k$ and $\ell$. These polynomials are generalizations of Ramanujan polynomials studied by  M. R. Murty, C. Smyth, and R. Wang.
		More specifically, we show that except for two real zeros, all other zeros of $R_{k,\ell}(z)$ lie on the unit circle. 
		
	\end{abstract}
	
	\maketitle
	
	\section{\bf Introduction}
	\bigskip
	
	The study of polynomials and their zeros is an important theme in mathematics, with relevance to a wide variety of areas. We call a polynomial $P(x)$ of degree $d$ to be \textit{reciprocal} if $x^d \, P(1/x) = \pm P(x)$. The main object of investigation in this paper is a two-parameter family of reciprocal polynomials related to values of the Riemann zeta-function at even positive integer arguments that arises as a generalization of Ramanujan polynomials introduced by M. R. Murty, C.
	Smyth, and R. Wang in \cite{msw}.  \\
	
	Recall that {the Riemann zeta function is defined by
		\[
		\zeta(s) := \sum_{n\geq 1} \frac{1}{n^{s}} \qquad \text{for} \qquad 
		\text{Re}(s) > 1,
		\]}
	
	and for $n \geq 0$, the Bernoulli numbers are given by
	\begin{equation*}
		\frac{t}{e^t - 1} = \sum_{n \geq 0} \frac{B_n}{n!} \, t^n \qquad |t| < 2 \pi.
	\end{equation*}
	From this, we have that $B_0 = 1$, $B_1 = -1/2$, $B_{2k+1} = 0$ for $k \geq 1$ and $B_{n} \in \mathbb{Q}$. In particular, we have Euler's formula, namely,
	\begin{equation}\label{zeta-value-eqn}
		\zeta(2k) = - \, \frac{{(2 \pi i)}^{2k} \, B_{2k}}{2 \, (2k) !} \qquad k \geq 1.
	\end{equation}
	Thus, since $\pi$ is transcendental, it is evident that the even zeta-values are transcendental. The transcendental nature of odd zeta-values however, remains unknown. \\
	
	An attempt towards this study is Ramanujan's formula for odd zeta-values (see \cite{r}), that is, for positive real numbers $\alpha$ and $\beta$ with $\alpha \beta = \pi^2$ and any positive integer $k$, 
	\begin{align*}
		\alpha^{-k} \left\{ \frac{1}{2} \, \zeta(2k+1) \, + \, \sum_{n=1}^{\infty} \frac{n^{-2k-1}}{e^{2\alpha n}-1} \right\} \, & - \, {\left( -\beta \right)}^{-k} \, \left\{ \frac{1}{2} \, \zeta ( 2k+1) \, + \, \sum_{n=1}^{\infty} \frac{n^{-2k-1}}{e^{2\beta n}-1} \right\} \\
		& \hspace{-3.5cm} = \, - \, 2^{2k} \, \alpha^{k+1} \, \sum_{j=0}^{k+1} {(-1)}^j \, \frac{B_{2j} \, B_{2k+2-2j}}{(2j)! \, (2k+2-2j)!} \, {\left( \frac{\beta}{\alpha} \right)}^j.
	\end{align*}
	
	A more general version of this formula was independently derived by E. Grosswald \cite{g} in 1970. Let $\mathbb{H}$ denote the upper half plane, that is the set of complex numbers with strictly positive imaginary part. For any positive integer $k$, set
	\begin{equation}\label{ramanujan-poly-def}
		R_{2k+1}(x) := \sum_{j=0}^{k+1} \frac{B_{2j} \, B_{2k+2-2j}}{(2j)! \, (2k+2-2j)!} \, x^{2j}
	\end{equation}
	Then we have the following.
	\begin{theorem*}[Grosswald]
		Let $k \geq 1$ be an integer and $\sigma_k(n) := \sum_{d \mid n} d^k$. For $z \in \mathbb{H}$, define
		\begin{equation*}
			\mathcal{F}_{k} (z) := \sum_{n =1}^{\infty} \frac{\sigma_k(n)}{n^k} \, e^{2 \pi i n z}.
		\end{equation*}
		Then 
		\begin{align}\label{grosswald-form-eqn}
			\mathcal{F}_{2k+1}(z) - z^{2k} \,\mathcal{F}_{2k+1} \left( - \frac{1}{z} \right) = \frac{1}{2} \, \zeta( 2k+1) \, \left( z^{2k} - 1 \right) + \frac{(2 \pi i)^{2k+1}}{2\, z} R_{2k+1}(z).
		\end{align}
	\end{theorem*}
	One can recover Ramanujan's formula from the above identity by specializing at $z=i\alpha/\pi$, $-1/z = i \beta / \pi$ and noting that
	\begin{equation*}
		\mathcal{F}_{k}(z) = - \zeta(k) - \sum_{n=1}^{\infty} \frac{n^{-k}}{e^{-2 \pi i n z}-1}.
	\end{equation*}
	The function $\mathcal{F}_{2k+1}(z)$ is an Eichler integral of the second kind of the standard Eisenstein series of weight $2k+2$ for $\text{SL}_2(\mathbb{Z})$, and has been well-studied in the literature. From \eqref{grosswald-form-eqn}, it is clear that if $z_0$ is a zero of $R_{2k+1}(x)$, which is not a $2k$-th root of unity, then we obtain an expression for the odd zeta value $\zeta(2k+1)$ in terms of special values of the Eichler integral $\mathcal{F}_{2k+1}(z_0)$. \\
	
	With this motivation, in 2011, M. R. Murty, C. Smyth, and R. Wang isolated the polynomials $R_{2k+1}(x)$ for independent study and called them \textit{`Ramanujan polynomials'}. It is not difficult to see that the Ramanujan polynomials are reciprocal polynomials, that is, $R_{2k+1}(x) = x^{2k+2} R_{2k+1}(1/x)$. In the same paper, they observed that most of the zeros of these polynomials lie on the unit circle. 
	\begin{theorem*}[M. R. Murty, C. Smyth, and R. Wang]
		For $k \geq 0$, all the non-real zeros of $R_{2k+1}(x)$ are simple and lie on the unit circle. For $k \geq 1$, $R_{2k+1}(x)$ has exactly four distinct real zeros. 
		These real zeros of $R_{2k+1}(x)$ can be expressed as $\alpha_k$, $-\alpha_k$, $1/\alpha_k$ and $- 1/\alpha_k$ for a real number $\alpha_k > 1$. Moreover, $\alpha_k \in (2,2.2)$ and $\alpha_k \rightarrow 2$ as $k \rightarrow \infty$.
	\end{theorem*}
	
	One may recognize the Ramanujan polynomial $R_{2k+1}(x)$ as essentially the odd-part of the period polynomial associated to the Eisenstein series of weight $2k+2$ for $\text{SL}_2(\mathbb{Z})$. This suggests that we can study period polynomials associated to modular forms and in particular to Hecke eigenforms. With this context, the above result gave rise to a flurry of activity in determining the location of zeros of period polynomials of Hecke eigenforms, and their odd parts. This theme is now referred to as the \textit{`Riemann hypothesis for period polynomials'}. The interested reader can refer to the excellent survey \cite{dr} for further details. \\
	
	Recently, several generalizations of Ramanujan's formula have appeared in the literature. In particular, A. A. Dixit and R. Gupta \cite[Theorem 2.1]{dg} obtained an analogous formula for the square of odd zeta values. 
	\begin{theorem*}[A. Dixit, R. Gupta]
		Let $\tau(r) = \# $divisors of $r$, $\epsilon = e^{i \pi / 4}$,  $\widetilde{\sigma}_s(n) := \sum_{d \mid n} \tau(d) \, \tau(n/d) \, d^{s},$ and
		\begin{align*}
			\mathcal{G}_{2k+1}(x) := \zeta^2(2k+1)   \bigg(  \gamma + & \log \left( \frac{x}{\pi}\right)  -  \frac{\zeta'(2k+1)}{\zeta(2k+1)} \bigg) \\
			& + \sum_{n=1}^{\infty} \frac{\widetilde{\sigma}_{2k+1}(n)}{n^{2k+1}} \, \bigg( K_0(4\epsilon\sqrt{n}x) + K_0(4\overline{\epsilon}\sqrt{n}x) \bigg) , \qquad x>0, \qquad
		\end{align*}
		where $K_z(x):= \frac{\pi}{2} \frac{I_{-z}(x) - I_z(x)}{\sin \pi z}$ is the modified Bessel function of the second kind of order $z$, and $I_z$ is that of the first kind. Then, for positive real number $\alpha$ and $\beta$ with $\alpha \beta = \pi^2$, 
		\begin{align*}
			\alpha^{-2k} \,  \mathcal{G}_{2k+1}(\alpha)  -   {(-1)}^k & \beta^{-2k} \, \mathcal{G}_{2k+1} (\beta) \\
			& = - \pi \, 2^{4k} \, \beta^{2k+2} \sum_{j=0}^{k+1} {(-1)}^j {\left( \frac{B_{2j}}{(2j)!} \, \frac{B_{2k+2-2j}}{(2k+2-2j)!} \right)}^2 \, {\left( \frac{\alpha}{\beta} \right)}^{2j}.
		\end{align*}
	\end{theorem*}
	This was further generalized to any positive integer power of odd zeta-values by S. Banerjee and V. Sahani \cite[Theorem 1.1]{bs} in 2023. \\
	
	Motivated by the analogy of the sums in terms of Bernoulli numbers appearing above with Ramanujan polynomials, and supported by numerical evidence, B. Maji and T. Sarkar in \cite{ms} proposed the following conjecture.
	\begin{conjecture}[Maji-Sarkar]\label{ms-conj}
		For positive integers $k$ and $\ell$, set
		\begin{equation*}
			R_{k,\ell}(x) := \sum_{j=0}^{k+1} {\left( -1\right)}^{(\ell+1)j} {\left( \frac{B_{2j}}{(2j)!} \, \frac{B_{2k+2-2j}}{(2k+2-2j)!} \right)}^{\ell} \, x^j.
		\end{equation*}
		Then all the non-real zeros of $R_{k,\ell}(x)$ lie on the unit circle.
	\end{conjecture}
	\noindent Note that $R_{k,1}(x^2) = R_{2k+1}(x)$, the Ramanujan polynomial. Thus, for $\ell=1$, the conjecture follows from the theorem of M. R. Murty, C. Smyth, and R. Wang \cite{msw}. \\
	
	It is easy to see that $R_{k,\ell}(x)$ is a reciprocal polynomial, that is $x^{k+1}R_{k,\ell}(1/x) = \pm R_{k,\ell}(x)$. 
	In this paper, we study the zeros of $R_{k,\ell}(x)$ and show that the above conjecture is true. Our main theorems can be summarized as follows.
	\begin{theorem}\label{main-thm-unit-circle}
		For positive integers $k$ and $\ell$, Conjecture \ref{ms-conj} is true. More specifically, except for two real zeros of the form 
		\begin{equation*}
			\alpha_{k,\ell} \qquad \text{ and } \qquad \frac{1}{\alpha_{k,\ell}}
		\end{equation*}
		with $\alpha_{k,\ell}>1$, the remaining zeros of $R_{k,\ell}(x)$ lie on the unit circle. Moreover, all the zeros are simple.
	\end{theorem}
	
	\noindent When $\ell$ is odd, we also determine the location of the largest real zero.
	\begin{theorem}\label{main-thm-real-l-odd}
		For positive integers $k$ and $\ell$, with $\ell$ - odd and $k \geq 3$, let $\alpha_{k,\ell}$ denote the largest real zero of $R_{k,\ell}(x)$. Then 
		\begin{equation*}
			\alpha_{k,\ell} \in  \bigg( \, \, \left(\frac{2 \, \zeta(2) \, \zeta(2k)}{\zeta(2k+2)}\right)^{\ell}, \, \, \big(2\zeta(2)\big)^{\ell} \, \bigg({\frac{2}{\zeta(2)}\left(1+\frac{3 \, d_{\ell}}{4^k} \right)}\bigg)
			\, \, \bigg),
		\end{equation*}
		where $d_{\ell}=\frac{(1.75)^{\ell}-1}{0.75}.$ 
	\end{theorem}
	
	A few remarks are in order. First of all, it must be noted that when $\ell$ is even, we can consider the related family of polynomials $\widetilde{R}_{k,\ell}(x):= R_{k,\ell}(-x)$. Theorem \ref{main-thm-unit-circle} immediately implies that all the non-real zeros of these polynomials, namely,
	\begin{equation*}
		\sum_{j=0}^{k+1} {\left( \frac{B_{2j} \, B_{2k+2-2j}}{(2j)! \, (2k+2-2j)!} \right)}^{\ell} \, x^j,
	\end{equation*}
	also lie on the unit circle, for $k \geq 3$ and $\ell \geq 1$. However, for technical reasons that will be evident in the proof, it is more convenient to prove Theorem \ref{main-thm-unit-circle} for $R_{k,\ell}(x)$ than for $\widetilde{R}_{k,\ell}(x)$. The method of proof of Theorem \ref{main-thm-unit-circle} is inspired by that in \cite{msw}. However, the corresponding estimates in our case require to ascertain the dependence on $\ell$, which adds an extra layer of difficulty. \\
	
	The questions regarding the location of the largest real zero of $R_{k,\ell}(x)$ when $\ell$ is even, as well as the roots of unity which may be zeros of $R_{k,\ell}(x)$ appear to be very subtle and involved in this generality. We relegate these investigations to future research.
	
	\medskip

	\section{\bf Preliminaries}
	
	\bigskip 
	
	This section is dedicated to deriving the estimates that will be crucial in our proofs.\\
	
	\noindent Instead of the polynomial $R_{k,\ell}(z)$, we work with its monic, even counterpart, 
	\begin{equation*}
		\mkl(z) = \frac{{\left( -1 \right)}^{(\ell+1)(k+1)} }{{\left(B_{2k+2}/(2k+2)! \right)}^{\ell}} \, R_{k,\ell}(z^2).
	\end{equation*}
	Using \eqref{zeta-value-eqn}, $\mkl(z)$ can be rewritten as
	\begin{equation}\label{eq:M_{k,l}}
		\mkl(z) =z^{2k+2} + {(-1)}^{(\ell+1) (k+1)} - 2^{\ell} \sum_{j=1}^{k} (-1)^{(\ell+1)(k+j)} \, {\left( \frac{\zeta(2j) \, \zeta(2k+2-2j) }{\zeta(2k+2)} \right)}^{\ell} \, z^{2j}.
	\end{equation}
	For simplicity, let
	\begin{equation*}
		q_j := \frac{\zeta(2j) \, \zeta(2k+2-2j) }{\zeta(2k+2)} {\qquad j=1,\, 2, \, \ldots, \, k},
	\end{equation*}
	with the dependence on $k$ being implicit. Thus,
	\begin{equation*}
		\mkl(z)
		=z^{2k+2} + (-1) ^{(\ell+1) (k+1)}
		- 2^{\ell}
		\sum_{j=1}^{k}
		(-1)^{(\ell+1) (k+j)} \, 
		q_j^{\ell} \, 
		z^{2j}.
	\end{equation*}
	The zeros of $R_{k,\ell}(z)$ are squares of the zeros of $M_{k,\ell}(z)$.\\

	Recall the following preliminary lemmas from \cite{msw}. 
	\begin{lemma}\label{bounds-lemma}
		We have the following inequalities for values of the Riemann zeta-function.
		\begin{itemize}
			\item[(a)] For $n \geq 2$, 
			\begin{equation}\label{lem:inequalities of zeta(n)}
				1+2^{-n}<\zeta(n)<1 \, + \, \left( \frac{n+1}{n-1}\right) \,  2^{-n}.
			\end{equation}
			\item[(b)] For $k\ge 1$ and {$j=1,\ldots,k$},
			\begin{equation}\label{lem:upper bound of quotient of zeta values}
				\frac{\zeta(2k+2-2j)}{\zeta(2k+2)}-1
				<3\cdot4^{j-(k+1)}.
			\end{equation}
			\item[(c)]  For $k\ge 3$ and $2\le j\le k-1$, 
			\begin{equation}\label{lem:max valu of frac{(2j+1)(2k+3-2j)}{(2j-1)(2k+1-2j)}}
				\frac{(2j+1)(2k+3-2j)}{(2j-1)(2k+1-2j)}
				< 25/9.
			\end{equation}
			\item[(d)] We have
			\begin{equation}\label{lem:identity of sum of even zeta values}
				\sum_{j=1}^{\infty}
				\frac{\zeta(2j)}{4^j}
				=\frac 12.
			\end{equation}
		\end{itemize}
	\end{lemma}
	\begin{proof}
		For proofs of parts (a), (b), and (d), see \cite[Lemma 4.4, Lemma 4.5 and Lemma 4.6]{msw}. The proof of part (c) follows along the same lines as that of \cite[Lemma 5.2]{msw}, but for $k \geq 3$ instead of $k \geq 5$.
	\end{proof}
	
	\begin{lemma}\label{lem:q_j decreasing}
		For $k\ge 2$ and $1\le j\le \floor{\frac{k+1}{2}}$, $q_j=\frac{\zeta(2j)\zeta(2k+2-2j)}{\zeta(2k+2)}$ is decreasing.
	\end{lemma}
	\begin{proof}
		For $k \geq 2$, let 
		\(
		f(x)
		=\zeta(2x)\zeta(2k+2-2x).
		\)
		Now 
		\[
		f'(x)
		=2f(x)
		\Bigg(
		\frac{\zeta'(2x)}{\zeta(2x)}
		-\frac{\zeta'(2k+2-2x)}{\zeta(2k+2-2x)}
		\Bigg)
		>0\;\; \text{for}\;\; x\in\left[\frac{k+1}{2},k\right],
		\]
		as \(\frac{\zeta'(x)}{\zeta(x)}\) is increasing for \(x>1\). This follows from the well-known Dirichlet series expansion
		\[
		\frac{\zeta'(x)}{\zeta(x)}
		= \, -\sum_{n=1}^{\infty}
		\frac{\Lambda(n)}{n^x},
		\] 
		where \(\Lambda(n) \) is the von Mangoldt function. Therefore, $f(x)$ is increasing in  $\left[\frac{k+1}{2},k\right]$. Hence, the result follows using $q_j=q_{k+1-j}$ and $\zeta(2k+2)>1$.
	\end{proof}
	
	\begin{lemma}\label{lem:(1+x)^l}
		For all $x\in [0,b]$ and $\ell \ge 1$, we have 
		\[
		(1+x)^{\ell}
		\le 1 \, + \, c_{\ell}(b) \,  x \qquad \text{ with } \qquad c_{\ell}(b)=\frac{(1+b)^{\ell}-1}{b}.
		\]
	\end{lemma}
	\begin{proof}
		Consider the function $f(x)=\frac{(1+x)^{\ell}-1}{x}$ on $ (0,b]$. Using calculus, it is easy to prove that $f(x)$ is a strictly increasing function on $(0,b]$. The result follows by taking $c_{\ell}(b)=f(b)$.
	\end{proof}
	
	\begin{lemma}\label{lem:(1+eps_j)^l}
		For each $k\ge 3$ and $\ell\ge 1$, let $2\le j\le k-1$ and
		\begin{equation}\label{epsilon-def-eqn}
			\varepsilon_{k,j}
			=\left( \frac{2j+1}{2j-1} \right) \, 4^{-j}
			+ \left( \frac{2k+3-2j}{2k+1-2j} \right) \, 4^{-k-1+j}
			+ \left( \frac{(2j+1)(2k+3-2j)}{(2j-1)(2k+1-2j)}\right) \,  4^{-k-1}.
		\end{equation}
		Then we have
		\[
		(1+\varepsilon_{k,j})^{\ell}
		\le 1 \, + \, c_{\ell} \, \varepsilon_{k,j} \quad \text{ where } \quad c_{\ell}=\frac{(1.306)^{\ell}-1}{0.306}.
		\]
	\end{lemma}
	\begin{proof}
		Treating this as a calculus problem involving a function of $j$, one can prove that 
		\[
		\max_{2\le j\le k-1} \varepsilon_{k,j}=\varepsilon_{k,2}.
		\]
		For $k\ge 3$, we have $\varepsilon_{k,2}\le 0.306$. Thus, the result follows from Lemma \ref{lem:(1+x)^l} by putting ${b=0.306}$. 
	\end{proof}
	
	\medskip 
	
	\section{\bf Real zeros of $R_{k,\ell}(z)$}
	\bigskip 
	
	We focus on understanding the properties of the real zeros of $R_{k,\ell}(z)$ in this section. 
	
	\begin{proposition}\label{plus-minus-1-as-zero-prop}
		Let $k$ and $\ell$ be positive integers with $k$-even. Then $R_{k,\ell}(1)=0$ if $\ell$ is even and $R_{k,\ell}(-1)=0$ if $\ell$ is odd.
	\end{proposition}
	\begin{proof}
		From \eqref{eq:M_{k,l}}, we see that
		\begin{equation*}
			R_{k,\ell}(z) = z^{k+1} + {(-1)}^{(k+1)(\ell +1)} - 2^{\ell} \sum_{j=1}^{k} {(-1)}^{(\ell+1)(k+j)} \, q_j^{\ell} \, z^j.
		\end{equation*}
		The claim is now evident since $q_j = q_{k+1-j}$.
	\end{proof}

	\noindent We gain further understanding about the zeros of $R_{k,\ell}(z)$ by considering the polynomial $M_{k,\ell}(z)$.
	\begin{proposition}\label{l-odd-real-zero-prop}
		Let $k$ and $\ell$ be positive integers with $\ell$ odd. Then, $\mkl(z)$ has exactly four real zeros, of the form $\beta_{k,\ell}$, $-\beta_{k,\ell}$, $1/\beta_{k,\ell}$ and $-1/\beta_{k,\ell}$ with $\beta_{k,\ell}>1$.
	\end{proposition}
	\begin{proof}
		For odd positive integer $\ell$, we have
		\[
		\mkl(z)           
		=z^{2k+2} + 1 
		-2^{\ell}
		\sum_{j=1}^{k}
		q_j^{\ell}
		z^{2j}.
		\] 
		Therefore $\mkl(1)
		=2-2^{\ell}
		\sum_{j=1}^{k}
		q_j^{\ell}
		<0$, as $q_j >1$ for  {$j=1,\ldots, k$}. Also, $\mkl(z)$ is positive for $z$ real, positive, and sufficiently large. Thus, by the intermediate value theorem, it has a real zero, $\beta_{k,\ell}$ say, greater than $1$. Since $M_{k,\ell}(z)$ is an even reciprocal polynomial, $-\beta_{k,\ell}$, $1/\beta_{k,\ell}$ and $-1/\beta_{k,\ell}$ are also zeros of $\mkl(z)$, and all of them are distinct. On the other hand, by Descartes’ Rule of Signs, we see that $\mkl(z)$ can have at most $2$ positive zeros. This completes the proof.
	\end{proof}
	
	In the next proposition, we study the real zeros of $\mkl(z)$ when $\ell$ is even. 
	\begin{proposition}\label{l-even-real-zero-prop}
		Let $k$ and $\ell$ be positive integers with $\ell$ even. 
		\begin{enumerate}
			\item[(a)] If $k$ is odd, then $\mkl(z)$ has at least four real zeros of the form $\beta_{k,\ell}$, $-\beta_{k,\ell}$, $1/\beta_{k,\ell}$ and $-1/\beta_{k,\ell}$ with $\beta_{k,\ell}>1$.
			\item[(b)] If $k$ is even, then $\mkl(z)$ has at least six real zeros of the form $-1$, $1$, $\beta_{k,\ell}$, $-\beta_{k,\ell}$, $1/\beta_{k,\ell}$ and $-1/\beta_{k,\ell}$ with $\beta_{k,\ell}>1$.
		\end{enumerate}
	\end{proposition}
	
	\begin{proof}
		For $k$ odd and $\ell$ even, 
		\[
		\mkl(z)           
		=z^{2k+2} + 1 
		+2^{\ell}
		\sum_{j=1}^{k}
		(-1)^j
		q_j^{\ell}
		z^{2j},
		\]
		which gives 
		\(
		\mkl(1)
		=2 
		+2^{\ell}
		\sum_{j=1}^{k}
		(-1)^j
		q_j^{\ell}. 
		\)
		Let 
		\(
		\gkl
		=\sum_{j=1}^{k}
		(-1)^j
		q_j^{\ell}
		.
		\)
		As \(k\) is odd and \(q_j=q_{k+1-j}\), we get
		\[
		\gkl
		=2 \left( \sum_{j=1}^{\frac{k-1}{2}}
		(-1)^j
		q_j^{\ell} \right)
		+(-1)^{\frac{k+1}{2}} \, 
		q_{\frac{k+1}{2}}^{\ell}.
		\]
		If  $k\equiv 1\pmod{4}$, we have
		\[
		\gkl
		= \, -2 \left[ \sum_{j=1}^{\frac{k-1}{4}}
		\bigg(
		q_{2j-1}^{\ell}
		-q_{2j}^{\ell}
		\bigg)  \right]          
		- q_{\frac{k+1}{2}}^{\ell},
		\]
		and if $k\equiv 3\pmod{4}$, we have
		\[
		\gkl
		= \, -2 \left[ \sum_{j=1}^{\frac{k-3}{4}}
		\bigg(
		q_{2j-1}^{\ell}
		-q_{2j}^{\ell}
		\bigg) \right] 
		-\bigg(
		q_{\frac{k-1}{2}}^{\ell}
		-q_{\frac{k+1}{2}}^{\ell}
		\bigg) 
		-q_{\frac{k-1}{2}}^{\ell}.
		\]
		In both the above cases, using Lemma \ref{lem:q_j decreasing} and the fact that $q_j>1$ for all {$j=1,\ldots,k$}, we obtain that $\gkl<-1$. Therefore 
		\(    
		\mkl(1)
		=2 
		+2^{\ell}
		\gkl  
		<0
		\)
		as {$\ell\ge2$}. Now $\mkl(z)$ is positive for $z$ real, positive, and sufficiently large. Thus, by the intermediate value theorem, it has a real zero, $\beta_{k,\ell}$ say, greater than $1$. As $\mkl(z)$ is an even reciprocal polynomial, $-\beta_{k,\ell}$ and $\pm 1/\beta_{k,\ell}$ are also distinct real zeros of $\mkl(z)$. This proves part (a). \\
		
		Now assume that $k$ is even. In this case, the polynomial under consideration becomes
		\[
		\mkl(z)           
		=z^{2k+2} - 1 
		-2^{\ell}
		\sum_{j=1}^{k}
		(-1)^j
		q_j^{\ell}
		z^{2j},
		\]
		which implies that
		\(
		\mkl(\pm 1)
		=0 
		\),
		by the symmetry $q_j=q_{k+1-j}$. Thus, we need to study the sign of the derivative of $\mkl(z)$ at $z=1$.\\
		
		\noindent From the above expression, we have
		\[
		\mkl'(z)
		=(2k+2)z^{2k+1}
		-2^{\ell+1}
		\sum_{j=1}^{k}
		(-1)^j \,
		j \, 
		q_j^{\ell} \, 
		z^{2j-1},
		\]
		which gives
		\begin{equation}\label{eq:mkl' and bkl}
			\mkl'(1)=(2k+2)(1-\hkl),
		\end{equation} 
		where
		$$
		\hkl
		=\frac{2^{\ell}}{k+1}
		\sum_{j=1}^{k}
		(-1)^j \,
		j \,
		q_j^{\ell}.
		$$
		As $k$ is even and $q_j=q_{k+1-j}$, we can rewrite this as
		\begin{equation*}\label{eq:bkl in sum}
			\hkl
			=\frac{2^{\ell}}{k+1}
			\sum_{j=1}^{k/2}
			(-1)^{j+1} \,
			(k+1-2j) \, 
			q_j^{\ell}.
		\end{equation*}
		
		\noindent If $k \equiv 0 \pmod 4$, we have
		\begin{equation*}\label{eq:bkl for k congruent to 0}
			\begin{split}
				\hkl
				=\frac{2^{\ell}}{k+1}
				\sum_{j=1}^{k/4}
				\Bigg(
				(k+3-4j) \, q_{2j-1}^{\ell}
				-(k+1-4j) \, q_{2j}^{\ell}
				\Bigg)
			\end{split}
		\end{equation*}     
		and if $k \equiv 2 \pmod 4$, we have
		\begin{equation*}\label{eq:bkl for k congruent to 2}
			\hkl
			=\frac{2^{\ell}}{k+1}
			\sum_{j=1}^{(k-2)/4}
			\Big(
			(k+3-4j)\, q_{2j-1}^{\ell}
			-(k+1-4j) \, q_{2j}^{\ell}
			\Big)
			+q_{\frac k2}^{\ell}
			.
		\end{equation*}
		In both cases, using Lemma \ref{lem:q_j decreasing} and the fact that $q_{\frac k2}>1$, we deduce that
		\begin{equation}\label{eq:lower bound of bkl}
			\hkl > {2}^{\ell} \, \left( \frac{k-3}{k+1} \right) \, \left( q_1^{\ell} - q_2^{\ell} \right) \, \geq \, 4 \,q_2^2 \, \left( \frac{k-3}{k+1} \right) \, \left( \left( \frac{q_1}{q_2} \right)^2 - 1 \right)
		\end{equation}
		as $\ell\ge 2$ and $q_1/q_2 > 1$. Using Lemma \ref{bounds-lemma}, equation \eqref{lem:inequalities of zeta(n)}, we have
		\begin{equation}\label{eq:lower bound of q_1/q_2}
			\frac{q_1}{q_2}=\frac{\zeta(2)\zeta(2k)}{\zeta(4)\zeta(2k-2)}
			\ge \left(\frac{\zeta(2)}{\zeta(4)} \right) \, \left(
			\frac{1+4^{-k}}{1+ \left(\frac{2k-1}{2k-3}\right) \,  4^{-k+1}} \right)
			=\frac{15}{\pi^2} \, \times \, \left( 
			\frac{1+4^{-k}}{1+ \left(\frac{2k-1}{2k-3}\right) 4^{-k+1}}\right) .     
		\end{equation}
		Using \eqref{eq:lower bound of q_1/q_2} in \eqref{eq:lower bound of bkl}, we get
		\[
		\begin{split}
			\hkl
			&> \, 4 \, q_2^2 \,
			\left( \frac{k-3}{k+1} \right) \,
			\left[
			\left( \frac{15}{\pi^2} \, \times \, 
			\frac{1+4^{-k}}{1+ \left(\frac{2k-1}{2k-3}\right) 4^{-k+1}}\right)^2
			-1
			\right]\\
			& > \, 4 \, \zeta(4)^2 \,
			\left( \frac{k-3}{k+1} \right) \,
			\left[
			\left( \frac{15}{\pi^2} \, \times \, 
			\frac{1+4^{-k}}{1+ \left(\frac{2k-1}{2k-3}\right) 4^{-k+1}}\right)^2
			-1
			\right]\\
			&>1.5
		\end{split}
		\]
		for $k\ge 4$, using the facts that
		\(
		\frac{k-3}{k+1}
		\)
		and 
		\(
		\frac{1+4^{-k}}{1+ \left(\frac{2k-1}{2k-3}\right) 4^{-k+1}} 
		\)
		are increasing functions of $k$. For $k=2$,  
		$$H_{2, \ell}=\frac{(2\zeta(2)\zeta(4))^{\ell}}{3\zeta(6)^{\ell}}>1$$ 
		for all $\ell$. Together with \eqref{eq:mkl' and bkl}, this implies that
		\(\mkl'(1)<0\) for all $k$ even and $\ell$ even. 
		
		\noindent Now, let $\pkl(z)=\frac{\mkl(z)}{z^2-1}$. Then all the zeros of $\mkl(z)$ other than $\pm 1$ are the zeros of $\pkl(z)$. Also, $\pkl(1)=\mkl'(1)<0$ and $\pkl(z)$ is positive for $z$ real, positive, and sufficiently large. Thus, by the intermediate value theorem, $\pkl(z)$ (and hence, $\mkl(z)$) has a real zero, $\beta_{k,\ell}$ say, greater than 1. As $\mkl(z)$ is an even reciprocal polynomial, $-\beta_{k,\ell}$ and $\pm 1/\beta_{k,\ell}$ are also zeros of $\mkl(z)$. This completes the proof.
	\end{proof}
	
	In order to show that these are the only real zeros of $M_{k,\ell}(z)$ when $\ell$ is even, we will prove that the remaining zeros must lie on the unit circle. This is achieved in the next section. \\
	
	We now give the proof of Theorem \ref{main-thm-real-l-odd}. This is equivalent to proving the following statement for the polynomial $\mkl(z)$.
	\begin{theorem}\label{real-zero-location-aux-thm}
		For  positive integers $k$ and $\ell$ with $\ell$ odd and $k \geq 2$, let $\beta_{k,\ell}$ denote the largest real zero of $\mkl(z)$. Then  
		$$ \beta_{k,\ell} \in 
		\Bigg(
		\left(\frac{2 \, \zeta(2) \, \zeta(2k)}{\zeta(2k+2)}\right)^{\ell/2}, \;\Big(2\zeta(2)\Big)^{\ell/2}\sqrt{\frac{2}{\zeta(2)}\Big(1+\frac{3 \, d_{\ell}}{4^k}}\Big)
		\Bigg),
		$$
		where 
		\(
		d_{\ell}
		=\frac{(1.75)^{\ell}-1}{0.75}.
		\)
	\end{theorem}
	\begin{proof}
		For $\ell$ odd and $q_1 = \frac{\zeta(2) \, \zeta(2k)}{\zeta(2k+2)}$, we get
		\[
		\mkl\Big(\Big(2q_1\Big)^{\ell/2}\Big)
		=\Big(2q_1\Big)^{\ell (k+1)}
		+1
		-2^{\ell}
		\sum_{j=1}^{k}
		q_j^{\ell} \, 
		\Big(2q_1\Big)^{\ell j}<0,
		\]
		as $q_k=q_1$ and $q_j>1$ for {$j=1,\ldots,k$}. On the other hand,
		\[
		\frac{\mkl(\sqrt{4+t})}{(4+t)^{k+1}}
		=1
		+\frac{1}{(4+t)^{k+1}}
		-2^{\ell}
		\sum_{j=1}^{k}
		q_j^{\ell} \, 
		(4+t)^{j-k-1}.
		\]
		Changing the variable $j$ to $k+1-j$ in the sum and using $q_j=q_{k+1-j}$, we get
		\[
		\frac{\mkl(\sqrt{4+t})}{(4+t)^{k+1}}
		=1
		+\frac{1}{(4+t)^{k+1}}
		-2^{\ell}
		\sum_{j=1}^{k}
		\frac{q_j^{\ell}}
		{(4+t)^{j}}.
		\]
		Now replacing $1$ by $2\sum_{j=1}^{\infty}{\zeta(2j)}/{4^j}$  using Lemma \ref{bounds-lemma}, equation \eqref{lem:identity of sum of even zeta values}, the above expression becomes
		\[
		\frac{\mkl(\sqrt{4+t})}{(4+t)^{k+1}}
		=2\sum_{j=k+1}^{\infty}\frac{\zeta(2j)}{4^j}
		+\frac{1}{(4+t)^{k+1}}
		+\sum_{j=1}^{k}
		\Bigg[
		\frac{2\zeta(2j)}{4^j}
		-\frac{(2q_j)^{\ell}}
		{(4+t)^{j}}
		\Bigg].     
		\]
		We want to find a $t$ for which $\mkl(\sqrt{4+t})>0$. Thus, it is sufficient to find $t$ for which
		\[
		\begin{split}
			&\frac{2\zeta(2j)}{4^j}
			>\frac{(2q_j)^{\ell}}
			{(4+t)^{j}}
			=\frac{(2\zeta(2j))^{\ell}}{(4+t)^{j}}
			\Bigg(
			\frac{\zeta(2k+2-2j)}{\zeta(2k+2)}
			\Bigg) ^{\ell}
			\;\;
			\text{for}\;\;       
			{ j=1,\ldots,k}.
		\end{split}
		\]
		Using the upper bound in Lemma \ref{bounds-lemma}, equation \eqref{lem:upper bound of quotient of zeta values}, the above inequality reduces to
		\[
		\Bigg(
		\frac{4+t}{4}
		\Bigg)^{j}
		>\lp 2\zeta(2) \rp ^{\ell-1}
		\Bigg(
		1+3\cdot4^{j-(k+1)}
		\Bigg)^{\ell}, \;\; \text{as}\;\; \zeta(2)>\zeta(2j)\;\;\text{for}\;\; {j=2,\ldots,k},
		\]
		or equivalently
		\[
		\Bigg(
		\frac{4+t}{4}
		\Bigg)^{j}
		>\lp 2\zeta(2) \rp ^{\ell-1}
		\Bigg(
		1+3 \, d_{\ell}\cdot4^{j-(k+1)}
		\Bigg),     
		\]
		where $d_{\ell}=\frac{(1.75)^{\ell}-1}{0.75}$ follows from Lemma \ref{lem:(1+x)^l} by putting $b=0.75$. Thus,
		\[
		\frac{4+t}{4}
		>\lp 2\zeta(2) \rp ^{\frac{\ell-1}{j}}
		\Bigg(
		1+3 \, d_{\ell}\cdot4^{j-(k+1)}
		\Bigg)^{1/j}. 
		\]
		For $a\ge 0$ and $0<\de\le 1$, we have $(1+a)^{\de}\le1+\de a$. we replace the above condition by 
		\[
		\frac{4+t}{4}
		>\lp 2\zeta(2) \rp ^{\frac{\ell-1}{j}}
		\Bigg(
		1
		+\frac{3 \, d_{\ell}\cdot4^{j-(k+1)}}{j}
		\Bigg). 
		\]
		The above lower bound is attained at $j=1$. Hence we obtain our final sufficient condition for $\mkl(\sqrt{4+t})$ to be positive, namely that
		\[
		4+t
		>2^{\ell+1}
		\zeta(2)^{\ell-1}
		\Bigg(
		1
		+\frac{3 \, d_{\ell}}{4^k}
		\Bigg).
		\]
		This proves the theorem.
	\end{proof}
	Note that in order to prove the upper bound on $\beta_{k,\ell}$, we did not use the parity of $\ell$. This bound also holds when $\ell$ is even.

	\medskip 
	
	\section{\bf Proof of Theorem \ref{main-thm-unit-circle}}
	
	\bigskip 
	
	To complete the proof of Theorem \ref{main-thm-unit-circle}, we now study the number of zeros of $R_{k,\ell}(z)$ on the unit circle. Once again, we do this by focusing on the polynomial $M_{k,\ell}(z)$ instead. The idea of the proof is to approximate $M_{k,\ell}(z)$ by a suitable trigonometric polynomial $A_{k,\ell}(z)$ and count the number of sign changes of the associated real-valued function on the unit circle.\\

	For positive integers $k$ and $\ell$ with $k \geq 3$, let 
	\[
	\akl(z)
	:=z^{2k+2} + (-1)^{(\ell+1)(k+1)}
	-(2q_1)^{\ell}
	\Bigg(
	z^{2k}+ (-1)^{(\ell+1)(k+1)}z^{2}
	\Bigg)\\
	- 2^{\ell}
	\sum_{j=2}^{k-1}
	(-1) ^{(\ell+1) (k+j)}
	z^{2j},
	\]
	and set 
	\begin{equation}\label{eq:def of Delta}
		\begin{split}
			\dkl(z)
			& := \mkl(z)-\akl(z)\\
			&= \, - \, 2^{\ell}
			\sum_{j=2}^{k-1}
			(-1) ^{(\ell+1) (k+j)} \,
			(q_j^{\ell}-1) \, 
			z^{2j}.
		\end{split}   
	\end{equation}
	
	\noindent We first establish that $A_{k,\ell}(z)$ is a `good' approximation for $M_{k,\ell}(z)$ on the unit circle.
	\begin{proposition}\label{delta-upper-bound-prop}
		For all $z \in \mathbb{C}$ such that $|z|=1$, 
		\begin{equation}\label{eq:abs of Delta}
			\big| {\dkl(z)} \big| < 2^{\ell} \, c_{\ell} \times 0.2762, \qquad  \text{with} \qquad 
			c_{\ell}=\frac{(1.306)^{\ell}-1}{0.306},
		\end{equation}
		for $k \geq 3$ and $\ell \geq 1$.
	\end{proposition}
	\begin{proof}
		We have that
		\begin{align*}
			\vt{\dkl(z)}
			&=\vt{ 
				-2^{\ell}
				\sum_{j=2}^{k-1}
				(-1) ^{(\ell+1) (k+j)}
				(q_j^{\ell}-1)
				z^{2j} 
			}\\
			&\le 
			2^{\ell}
			\sum_{j=2}^{k-1}     
			(q_j^{\ell}-1).
		\end{align*}
		Lemma \ref{bounds-lemma}, equation \eqref{lem:inequalities of zeta(n)} gives
		\begin{equation}
			\begin{split}
				q_j &<\zeta(2j)\zeta(2k+2-2j)\\
				&<1+\left(\frac{2j+1}{2j-1}\right) \, 4^{-j}
				+ \left(\frac{2k+3-2j}{2k+1-2j}\right) \, 4^{-k-1+j}
				+ \left(\frac{(2j+1)(2k+3-2j)}{(2j-1)(2k+1-2j)} \right) 4^{-k-1}\\
				&=1+\varepsilon_{k,j},
			\end{split}  
		\end{equation}
		where $\varepsilon_{k,j}$ is as in \eqref{epsilon-def-eqn}. Thus,
		\[
		\vt{\dkl(z)}
		<2^{\ell}
		\sum_{j=2}^{k-1}     
		\big( (1+\varepsilon_{k,j})^{\ell}-1 \big).
		\]
		Applying Lemma \ref{lem:(1+eps_j)^l}, we get that for $k\ge 3$,
		\[
		\begin{split}
			\vt{\dkl(z)}    
			&<2^{\ell} c_{\ell} \, 
			\sum_{j=2}^{k-1}
			\varepsilon_{k,j}\\
			&=2^{\ell} c_{\ell} \,
			\Bigg(
			2 \sum_{j=2}^{k-1}
			\left(\frac{2j+1}{2j-1} \right) \, 4^{-j} \,
			+ \, \sum_{j=2}^{k-1}
			\left(\frac{(2j+1)(2k+3-2j)}{(2j-1)(2k+1-2j)}\right) \, 4^{-k-1}
			\Bigg),
		\end{split} 
		\]
		where $c_{\ell}=\frac{(1.306)^{\ell}-1}{0.306}$. Now using the bounds from Lemma \ref{bounds-lemma}, equation \eqref{lem:max valu of frac{(2j+1)(2k+3-2j)}{(2j-1)(2k+1-2j)}}, we get
		\begin{equation}
			\begin{split}
				\vt{\dkl(z)}
				&<2^{\ell} c_{\ell} \,
				\left(
				2 \left[ \sum_{j=2}^{\infty}
				\Bigg(
				1+
				\frac{2}{2j-1} 
				\Bigg) 4^{-j} \right]
				+ \frac{25 \, (k-2) \, 4^{-k-1}}{9}
				\right)\\
				&=2^{\ell} c_{\ell} \, 
				\Bigg(
				\log 3-\frac 56
				+ \frac{25 \, (k-2) \, 4^{-k-1}}{9}
				\Bigg),
			\end{split}
		\end{equation}
		since
		\[
		1
		+4 \sum_{j=2}^{\infty}  
		\frac{1}{2j-1} 
		4^{-j} = \log 3.
		\]
		The claim follows by taking $k \geq 3$.    
	\end{proof}

	\noindent We now investigate the behaviour of $\mkl(z)$ and $\akl(z)$ on the unit circle, by setting $z=e^{i\th}$. For $\th \in [0, \, 2\pi]$, let 
	\begin{equation*}
		F_{k,\ell}(\th)=
		\begin{cases}
			-i \, e^{-i(k+1)\th} \mkl(e^{i\th}) & \text{ if } k \text{ and } \ell \text{ are both even},\\
			e^{-i(k+1)\th} \mkl(e^{i\th}) & \text{ otherwise.}
		\end{cases}
	\end{equation*}
	The reciprocal nature of $\mkl(z)$ implies that the function $F_{k,\ell}(\th)$ is real-valued. Similarly, let
	\begin{equation*}
		\fkl(\th) =
		\begin{cases}
			-i \, e^{-i(k+1)\th} \, \akl(e^{i \th}) & \text{ if } k \text{ and } \ell \text{ are both even},\\
			e^{-i(k+1)\th} \, \akl(e^{i \th}) & \text{ otherwise.}
		\end{cases}
	\end{equation*}
	By the symmetry of coefficients, it is evident that $\fkl(\th)$ is also real-valued. Proposition \ref{delta-upper-bound-prop} is equivalent to the statement that 
	\begin{equation}\label{eqn-diff-F-f-bound}
		| F_{k,\ell}(\th) - \fkl(\theta)| < 2^{\ell} \, c_{\ell} \, \times \, 0.2762, \quad  \text{with} \quad 
		c_{\ell}=\frac{(1.306)^{\ell}-1}{0.306}.
	\end{equation}
	
	\noindent We use this to prove the theorem below.
	\begin{theorem}\label{aux-thm-complex-zeros}
		For positive integers $k$ and $\ell$ with $k \geq 3$, the polynomial $\mkl(z)$ has at least $2k-2$ zeros on the unit circle.
	\end{theorem}
	\begin{proof}
		It suffices to show that real-valued function $F_{k,\ell}(\theta)$ changes sign at least $2k-1$ times in $[0, \, 2\pi]$. We establish this in three different cases, depending on the parity of $\ell$ and $k$. \\
		
		\noindent First, note that $\akl(z)$ can be expressed as
		\begin{align*}
			\akl(z) & = z^{2k+2} + (-1)^{(\ell+1)(k+1)}
			-(2q_1)^{\ell}
			\Bigg(
			z^{2k}+ (-1)^{(\ell+1)(k+1)}z^{2}
			\Bigg) \\
			& \hspace{5.5cm} - (-1)^{k(\ell+1)} \,
			2^{\ell} \, 
			z^4 \, 
			\sum_{j=0}^{k-3}
			{\left( (-1) ^{(\ell+1)}
				z^{2} \right)}^j \\
			& = z^{2k+2} + (-1)^{(\ell+1)(k+1)}
			-(2q_1)^{\ell}
			\Bigg(
			z^{2k}+ (-1)^{(\ell+1)(k+1)}z^{2}
			\Bigg) \\ 
			& \hspace{5.5cm} - (-1)^{k(\ell+1)} \,
			2^{\ell} \,
			z^4 \,
			\left(\frac{(-1)^{(\ell+1)(k-2)}z^{2k-4}-1}{\left((-1)^{(\ell +1)}z^2\right)-1}\right).
		\end{align*}
		Multiplying both side by $z^{-(k+1)}$ gives
		\begin{align*}
			z^{-(k+1)}\akl(z)
			& =z^{k+1}
			+(-1)^{(\ell+1)(k+1)} z^{-(k+1)}
			-(2q_1)^{\ell}
			\Bigg(
			z^{k-1}+ (-1)^{(\ell+1)(k+1)}z^{-(k-1)}
			\Bigg)\\
			& \hspace{4.5cm} - (-1)^{k(\ell+1)} \,
			2^{\ell} \,
			\left( \frac{(-1)^{(\ell+1)(k-2)}z^{k-2} - z^{-(k-2)}}{\left({(-1)}^{\ell + 1}z \right) - z^{-1}} \right).
		\end{align*}
		
		\begin{enumerate}
			\item[(i)] $\ell$ is odd: The function $\fkl(\th)$ can be simplified using the standard trigonometric identities as follows. 
			\begin{align*}
				\fkl(\th)
				&=2\cos(k+1)\th
				-2 \, (2q_1)^{\ell} \,
				\cos(k-1)\th
				-2^{\ell} \, \left(
				\frac{\sin(k-2)\th}{\sin \th} \right)\\
				&=2 \, 
				\bigg(
				\cos(k-1)\th \, \cos 2\th
				- \sin(k-1)\th \, \sin 2\th
				\bigg)
				-2 \, (2q_1)^{\ell} \,
				\cos(k-1)\th\\
				& \hspace{6cm}-2^{\ell} \,
				\bigg(
				\sin(k-1)\th \, \cot \th
				-\cos(k-1)\th
				\bigg)\\
				&=\cos(k-1)\th \,
				\left( 
				2\cos 2\th
				-2(2q_1)^{\ell}
				+2^{\ell}
				\right)
				-\sin(k-1)\th \,
				\left(
				2\sin 2\th
				+2^{\ell}
				\cot \th
				\right).
			\end{align*}
			Choose $\th_j=\frac{j\pi }{k-1}$, for {$j =1,\ldots, 2k-3$}, $j \neq k-1$. Then $0< \th_j <2\pi$ and
			\begin{equation}\label{eq:flk at theta_j for l odd}
				\fkl(\th_j)
				=(-1)^j
				\Bigg(
				2\cos \lp \frac{2\pi j}{k-1} \rp
				-2(2q_1)^{\ell}
				+2^{\ell}
				\Bigg).
			\end{equation}
			Note that for each {$j=1, \ldots, 2k-3$} and $j \neq k-1$, 
			\begin{equation}\label{eq:sign of fn for l odd}
				2\cos \lp \frac{2\pi j}{k-1} \rp
				-2(2q_1)^{\ell}
				+2^{\ell}
				\,\, < \,\, 2
				-2(2q_1)^{\ell}
				+2^{\ell}
				\,\, < \,\, 0,
			\end{equation}
			since $q_1>1$ and $\ell\ge 1$.
			Using \eqref{eq:sign of fn for l odd} in \eqref{eq:flk at theta_j for l odd}, we deduce that
			$\fkl(\th_j)$ has sign $(-1)^{j+1}$ for $0 < j\le 2k-3$. Additionally, for $0 < j\le 2k-3$, $j \neq k-1$,
			\[
			\abs{\fkl (\th_j)}
			=\vt
			{
				2\cos \lp \frac{2\pi j}{k-1} \rp
				-2(2q_1)^{\ell}
				+2^{\ell}
			}
			\,\, \ge \,\,  2(2q_1)^{\ell}
			-2^{\ell} -2,
			\]
			which can easily be checked to be 
			\begin{equation*}
				> 2^{\ell} \, c_{\ell} \, \times \, 0.2762,
				\quad  \text{with} \quad 
				c_{\ell}=\frac{(1.306)^{\ell}-1}{0.306}
			\end{equation*}
			as $q_1 > \zeta(2) > 1.644$. Therefore, by Proposition \ref{delta-upper-bound-prop}, we obtain that $F_{k,\ell}(\theta_j)$ has sign $(-1)^{j+1}$ for $0 < j < 2k-3$ and $j \neq k-1$. Moreover, if $\theta_0 := 0$ and $\theta_{k-1} := \pi$, that is $z=\pm1$, then we have already shown in the proof of Proposition \ref{l-odd-real-zero-prop} that $M_{k,\ell}(1) = M_{k,\ell}(-1) < 0$. Thus, the continuous function $F_{k,\ell}(\theta)$ has sign $(-1)^{j+1}$ for $0 \leq j \leq 2k-3$, and hence, changes sign at least $2k-1$ times in $[0, \,2\pi]$, which proves the result. \\
			
			\item[(ii)] $\ell$ is even and $k$ is odd: In this case, the function $\fkl(\th)$ becomes
			\begin{align*}
				\fkl(\theta) = \cos(k-1)\th \,
				\bigg(
				2\cos 2\th
				-2(2q_1)^{\ell}
				+2^{\ell}
				\bigg)
				\, \, - \, \, \sin(k-1)\th \,
				\bigg(
				2\sin 2\th
				-2^{\ell}
				\tan \th
				\bigg).
			\end{align*}
			Once again, choose $\th_j=\frac{j\pi }{k-1}$, {$j=0,\ldots, 2k-3$}. For  $j \neq (k-1)/2$, $3(k-1)/2$, using a similar argument as above, we conclude that $f_{k,\ell}(\theta_j)$ has sign $(-1)^{j+1}$ at the above values of $j$. For $\theta_{(k-1)/2} = \pi / 2$, we have 
			\begin{align*}
				\fkl(\pi/2) & = \cos \left( \frac{(k-1) \pi}{2} \right) \, \bigg( 2^{\ell} - 2 {(2 q_1)}^{\ell} -2 \bigg) + 2^{\ell} \, \lim_{\theta \rightarrow \pi/2} \frac{\sin (k-1)\theta}{\cos \theta} \\
				& = \, - \, \cos \left( \frac{(k-1) \pi}{2} \right)  \,  \bigg( 2 (2q_1)^{\ell} - 2^{\ell} +2 \bigg) - 2^{\ell}  \, \left( \frac{(k-1) \, \cos \frac{(k-1)\pi}{2}}{\sin \frac{\pi}{2}} \right) \\
				& = \, {(-1)}^{\left(\frac{k-1}{2}\right)+1} \, \bigg[  \bigg( 2 (2q_1)^{\ell} - 2^{\ell} +2 \bigg) \, + \, 2^{\ell} (k-1) \bigg],
			\end{align*}
			which has the same sign as ${(-1)}^{\left(\frac{k-1}{2}\right)+1}$ as the term inside square brackets is positive. Similarly, we have for $\theta_{3(k-1)/2} = 3 \pi /2$, we get
			\begin{align*}
				\fkl(3\pi/2) 
				=  \, {(-1)}^{\left(\frac{3(k-1)}{2}\right)+1} \, \bigg[ \bigg( 2 (2q_1)^{\ell} - 2^{\ell} +2 \bigg) + 2^{\ell} (k-1) \bigg] ,
			\end{align*}
			which has the same sign as $(-1)^{\left(\frac{3(k-1)}{2}\right)+1}$. Therefore, by Proposition \ref{delta-upper-bound-prop}, we obtain that $F_{k,\ell}(\theta)$ changes sign at least $2k-1$ times in $[0,\, 2\pi]$.\\
			
			\item[(iii)] $\ell$ and $k$ are both even: In this scenario, the simplified form of $\fkl(\th)$ is 
			\begin{align*}
				\fkl(\th) =\sin(k-1)\th \,
				\bigg(
				2\cos 2\th
				-2(2q_1)^{\ell}
				+2^{\ell}
				\bigg)
				\,\, + \,\, \cos(k-1)\th \,
				\bigg(
				2\sin 2\th
				-2^{\ell}
				\tan \th
				\bigg).
			\end{align*}
			Here, we choose $\th_j=\frac{(2j-1)\pi}{2(k-1)}$,{ $j=1,\ldots, 2k-2$}. For $j \neq k/2$ and $(3k/2)-1$, earlier arguments go through verbatim and imply that $f_{k,\ell}(\theta_j)$ has sign ${(-1)}^{j}$. At $j = k/2$, $\theta_{k/2} = \pi/2$ and 
			\begin{align*}
				\fkl(\pi/2) & = \, - \, \sin \left( \frac{(k-1)\pi}{2} \right) \bigg( 2{(2q_1)}^{\ell} - 2^{\ell} + 2 \bigg) - 2^{\ell} \lim_{\theta \rightarrow \pi/2} \frac{\cos (k-1)\theta}{\cos \theta} \\
				& = \, - \, \sin \left( \frac{(k-1)\pi}{2} \right) \bigg( 2{(2q_1)}^{\ell} - 2^{\ell} + 2 \bigg) - 2^{\ell}  \, \left( \frac{(k-1) \, \sin \frac{(k-1)\pi}{2}}{\sin \frac{\pi}{2}} \right) \\
				& = \, {(-1)}^{\frac{k}{2}} \, \bigg[ \bigg( 2 (2q_1)^{\ell} - 2^{\ell} +2 \bigg) + 2^{\ell} (k-1) \bigg],
			\end{align*}
			which has the same sign as ${(-1)}^{\frac{k}{2}}$. Likewise, for $j = (3k/2)-1$ and $\theta_{(3k/2)-1}= 3\pi/2$, we have
			\begin{align*}
				\fkl(3\pi/2) 
				= \, {(-1)}^{\frac{3k}{2}-1} \, \bigg[ \bigg( 2 (2q_1)^{\ell} - 2^{\ell} +2 \bigg) + 2^{\ell} \, 3 \, (k-1) \bigg],
			\end{align*}
			which has the sign ${(-1)}^{\frac{3k}{2}-1}$. Therefore, Proposition \ref{delta-upper-bound-prop} implies that $F_{k,\ell}(\theta)$ has at least $2k-1$ sign changes in $[0, \, 2\pi]$.
		\end{enumerate}
		This proves the claim.
	\end{proof}
	
	\noindent We put everything together to complete the proof of Theorem \ref{main-thm-unit-circle} below.
	\begin{proof}[Proof of Theorem \ref{main-thm-unit-circle}]
		Note that all zeros of $R_{k,\ell}(z)$ arise as squares of zeros of $M_{k,\ell}(z)$. First consider $k=1$, so that $R_{1,\ell}(z)$ is a quadratic polynomial. If $\ell$ is odd, using Proposition \ref{l-odd-real-zero-prop}, we see that $R_{1,\ell}(z)$ must have two real zeros of the form $\alpha_{1,\ell} = \beta_{1,\ell}^2$ and $\alpha_{k,\ell}^{-1} = \beta_{1,\ell}^{-2}$. If $\ell$ is even, Proposition \ref{l-even-real-zero-prop}(a) implies that $R_{k,\ell}(z)$ must have at least two real zeros, $\alpha_{1,\ell} = \beta_{1,\ell}^2$ and $\alpha_{1,\ell}^{-1} = \beta_{1,\ell}^{-2}$. Since $\deg R_{1,\ell}(z) = 2$, this proves the claim. \\
		
		Now suppose that $k=2$. By Proposition \ref{plus-minus-1-as-zero-prop}, we know that $+1$ or $-1$ is a zero of $R_{2,\ell}(z)$ according to whether $\ell$ is even or odd. Furthermore, Proposition \ref{l-odd-real-zero-prop} and Proposition \ref{l-even-real-zero-prop}(b) together imply that $R_{2,\ell}(z)$ has at least two real zeros of the form $\alpha_{2,\ell} = \beta_{2,\ell}^2 > 1$ and $\alpha_{2,\ell}^{-1} = \beta_{2,\ell}^{-2}$. As $\deg R_{2,\ell}(z) = 3$, the above argument proves the theorem when $k=2$.\\
		
		For $k \geq 3$, from Proposition \ref{l-odd-real-zero-prop} and \ref{l-even-real-zero-prop}, we deduce that $R_{k,\ell}(z)$ has at least two positive real zeros of the form $\alpha_{k,\ell}$ and $1/\alpha_{k,\ell}$. Combining this with Theorem \ref{aux-thm-complex-zeros} and comparing degrees completes the proof of the result about the location as well as simplicity of the zeros. 
	\end{proof}
	
	\medskip 
	
	\section{\bf Concluding remarks}
	
	\bigskip 
	
	Let $f(x) = a_0 \, x^m + a_1 \, x^{m-1} + \cdots + a_m \in \mathbb{C}[x]$ with $a_0 \neq 0$. Then the discriminant of $f$ is $\text{Disc}(f):= a_0^{2m-2} \prod_{1 \leq i < j \leq m} {\left( \alpha_i - \alpha_j \right)}^2$, where {$\alpha_1$, $\alpha_2$, $\ldots$, $\alpha_m$} are all the zeros of $f(x)$. A related quantity is the Mahler measure of $f$ and is given by
	\begin{equation*}
		M(f) := |a_0| \, \prod_{j=1}^m \max \left\{1, \, |\alpha_j| \right\}. 
	\end{equation*}
	Then Mahler \cite{mah} showed that
	\begin{equation*}
		|\text{Disc}(f)| \leq m^m \, M(f)^{2m-2}.
	\end{equation*}
	In case of the polynomials $R_{k,\ell}(x)$, Theorem \ref{main-thm-unit-circle} implies that
	\begin{equation*}
		M \left( R_{k,\ell} \right) = |a_{0,k,\ell}| \, |\alpha_{k,\ell}|,
	\end{equation*}
	where $a_{0,k,\ell}$ is the leading coefficient of $R_{k,\ell}(x)$ and $\alpha_{k,\ell}$ is its largest real zero. Therefore, we deduce that
	\begin{equation*}
		{\left( {(k+1)}^{-(k+1)} \, \prod_{1 \leq i < j \leq k+1} {\left( \alpha_{i,k,\ell} - \alpha_{j,k,\ell} \right)}^2 \right)}^{1/2k} \leq |\alpha_{k,\ell}|.
	\end{equation*}
	Thus, using the remark after Theorem \ref{real-zero-location-aux-thm}, we obtain that the largest real zero $\alpha_{k,\ell}$ always lies in the interval
	\begin{equation*}
		 \left(  \ {\left( {(k+1)}^{-(k+1)} \, \prod_{1 \leq i < j \leq k+1} {\left( \alpha_{i,k,\ell} - \alpha_{j,k,\ell} \right)}^2 \right)}^{1/2k}, \, \, \,  2^{\ell + 1} \, \zeta(2)^{\ell - 1} \, \left( 1 + \frac{3 \, d_{\ell}}{4^k}\right) \right),
	\end{equation*}
	where 
	\(
	d_{\ell}
	=\frac{(1.75)^{\ell}-1}{0.75}.
	\)
	Here $\alpha_{j,k,\ell}$ denote the zeros of $R_{k,\ell}(x)$. However, determining the size of the lower bound above does not appear straightforward. \\
	
	On another note, since the Bernoulli numbers are rational, the zeros of $R_{k,\ell}(z)$ are algebraic numbers, and their algebraic properties warrant further study. For instance, the irreducibility of $R_{k,\ell}(z)$ over $\mathbb{Q}$ when $k$ is odd and that of $R_{k,\ell}(z)/(z - {(-1)}^{\ell})$ when $k$ is even, is a natural question. Numerical evidence suggests that this is indeed the case. Note that a proof of such a fact would entail the irrationality of the real zero $\alpha_{k,\ell}$ of $R_{k,\ell}(z)$. A related problem is to identify whether any of the zeros of $R_{k,\ell}(z)$ are roots of unity. Preliminary computations suggest that except for $\pm 1$, the other zeros of $R_{k,\ell}(z)$ are not roots of unity for $\ell > 1$, but a proof of this fact seems elusive.\\
	
	A real algebraic integer $\alpha>1$, all of whose Galois conjugates except $\alpha$ lie inside $|z| \leq 1$, with at least one on the boundary $|z|=1$, is called a {\it Salem number}. These numbers are well-studied in the context of Diophantine approximation and related areas in number theory (see the survey \cite{s}). If we assume the irreducibility of $R_{k,\ell}(z)$ alluded to earlier, then for $k \geq 4$, the real zero $\alpha_{k,\ell}$ has properties similar to that of a Salem number, but is an algebraic \textit{number} instead of an algebraic \textit{integer}. It may be of interest to explore whether such numbers also possess significant Diophantine properties. We relegate the study of all these questions to future work. \\

	\medskip 
	
	\section*{Acknowledgments}
	We thank Prof. M. Ram Murty and the referees for helpful comments on an earlier version of this paper.

\end{document}